\documentclass{article}

\usepackage{amsmath,amsthm,amssymb}
\usepackage[top=2cm,bottom=2cm,left=2cm,right=2cm]{geometry}

\newtheorem{thm}{Theorem}

\newtheorem{lemma}{Lemma}

\begin{document}

\title{Space of initial conditions for a cubic Hamiltonian system}

\author{Thomas Kecker}
\date{}

\maketitle

\begin{abstract}
\noindent
In this paper we perform the analysis that leads to the space of initial conditions for the Hamiltonian system $q' = p^2 + zq + \alpha$, $p' = -q^2 - zp - \beta$, studied by the author in a previous article \cite{kecker1}. By compactifying the phase space of the system from $\mathbb{C}^2$ to $\mathbb{CP}^2$ three base points arise in the standard coordinate charts covering the complex projective space. Each of these is removed by a sequence of three blow-ups, a construction to regularise the system at these points. The resulting space, where the exceptional curves introduced after the first and second blow-up are removed, is the so-called Okamoto's space of initial conditions for this system which, at every point, defines a regular initial value problem in some coordinate chart of the space. The solutions in these coordinates will be compared to the solutions in the original variables.
\end{abstract}

\section{Introduction}
When studying the solutions of a differential equation in the complex plane a natural question to ask is what types of singularities can occur by analytic continuation of a local analytic solution, which exists by Cauchy's local existence and uniqueness theorem around every point where the equation is defined as a regular initial value problem. Points where the equation is itself singular are called fixed singularities of the equation. All other singularities of solutions, which arise somewhat spontaneously, i.e.\ they cannot be read off from the equation itself, are called movable singularities, as in fact their position varies with the initial conditions prescribed for the equation. Some differential equations are very special in this respect as the only movable singularities that can occur by analytic continuation of a local analytic solution are poles: an equation of this type is said to have the Painlev\'e property. In the class of second-order ordinary differential equations the six Painlev\'e equations stand out as nonlinear equations with this property. The six Painlev\'e equations are well-studied by many authors and have a rich mathematical structure of families of rational and special function solutions, B\"acklund transformations, and relations to integrable systems and isomonodromic deformation problems of associated linear systems. Another feature of the Painlev\'e equations was studied by Okamoto \cite{okamoto}: the so-called space of initial conditions. This space is constructed to deal with the equations at the points at infinity of the phase space, in particular at the movable singularities of the solutions. It is obtained by extending the phase space of the variables $(y,y')$ of the system to a compact space including the points where $y$ or $y'$ (or both) become infinite, e.g. $\mathbb{CP}^2$. Okamoto \cite{okamoto} constructed the space of initial conditions for all six Painlev\'e equations. Starting from so-called Hirzebruch surfaces (certain compact rational surfaces), each equation re-written in local coordinates on the surface can be regularised by a sequence of $8$ {\it blow-ups} (one needs $9$ blow-ups if starting from $\mathbb{CP}^2$). He showed that at every point of the resulting space, the equation written in some local coordinate chart in a neighbourhood of this point, forms a regular initial value problem. 

All six Painlev\'e equations can be written in the form of equivalent Hamiltonian systems. Sakai \cite{sakai} has studied the spaces of initial conditions for all six Painlev\'e equations in Hamiltonian form from a geometric viewpoint based on the symmetries of the underlying rational surfaces. The space of initial conditions for the Painlev\'e equations in Hamiltonian form was also studied in \cite{matano}. Joshi et.\ al.\ study the asymptotic behaviour of the solutions of the first Painlev\'e equation \cite{joshi1}, second Painlev\'e equation \cite{joshi2} and fourth Painlev\'e equation \cite{joshi3} via the space of initial conditions in the limit when the independent variable goes to infinity.

In this article we construct the space of initial conditions for the cubic Hamiltonian system (\ref{hamsys}) below, which was previously studied in \cite{kecker1} and is fact related to the Painlev\'e IV equation, but different to the standard Hamiltonian system considered e.g. in \cite{sakai} or \cite{matano}. Although the space of initial conditions is isomorphic to the one obtained there, the blow-up calculations in local coordinates are somewhat easier because of a $3$-fold symmetry in the singularity structure of the solutions: all singularities are simple poles with residues given by third roots of unity $1,\omega,\bar{\omega}$, where $\omega = \frac{-1+i\sqrt{3}}{2}$. After compactifying the phase space of the system to $\mathbb{CP}^2$ one finds three points where the system is indeterminate, the so-called {\it base points}, each of which can be resolved by a sequence of three blow-ups. Thus, the total number of blow-ups needed to regularise the system is also $9$ as in the case of Okamoto. However, the calculations to resolve each of these three base points are essentially the same up to factors of $\omega$ and $\bar{\omega}$. Each blow-up extends the phase space by introducing an additional line, called the exceptional curve. We show that the exceptional curves arising from the first two blow-ups are repellors of the dynamic system, meaning that the only singularities that can arise occur when the solution passes through the exceptional curve after the third blow-up. On this line the system of equations forms a regular initial value problem, which is a manifestation of the Painlev\'e property of this system. Thus the space of initial conditions, formed by the extended compact phase space after the three blow-ups for each base point, with the exceptional curves after the first and second blow-up removed, has the property that at each point there exist local coordinates such that the system of equations has a local analytic solution. We will see how the solutions of the system after the last blow-up are related to the solutions in the original variables.

The regular systems obtained after the third blow-up of each base point were also used in \cite{kecker1} for an alternative proof of the Painlev\'e property for the Hamiltonian system (\ref{hamsys}). The solutions of system (\ref{hamsys}) were studied further by Steinmetz \cite{steinmetz} by the re-scaling method, in particular their asymptotic behaviour in sectors of the complex plane and the distribution of poles.

\section{A cubic Hamiltonian system}
We consider the following Hamiltonian system, introduced in \cite{kecker1}, with Hamiltonian cubic in $p$ and $q$,
\begin{equation*}
H(z,p,q) = \frac{1}{3} \left( p^3 + q^3 \right) + zpq + \alpha p + \beta q,
\end{equation*}
the Hamiltonian equations being
\begin{equation}
\label{hamsys}
q' = \frac{\partial H}{\partial p} = p^2 + z q + \alpha, \quad p' = -\frac{\partial H}{\partial q} = -q^2 - z p - \beta.
\end{equation}
This system is related to the fourth Painlev\'e equation in the following way. Namely, the combination $w = p + q - z$ satisfies the equation
\begin{equation}
\label{p4scal}
2ww'' = w'^2 - w^4 -4zw^3 -(2\alpha+2\beta+3z^2)w^2 - (1-\alpha+\beta)^2,
\end{equation}
which becomes $P_{I\!V}$ after a simple rescaling of variables. Furthermore, the combinations $w = \rho p + \bar{\rho}q - z$, $\rho \in \{\omega,\bar{\omega}\}$, satisfy the same equation (\ref{p4scal}) but with the parameters $\alpha$ and $\beta$ replaced by $\rho \alpha$, $\bar{\rho} \beta$. Hence, by linear combination of solutions of equation (\ref{p4scal}) with different parameters, the solutions of system (\ref{hamsys}) can be expressend completely in terms of the fourth Painlev\'e transcendents. 

Although it is already granted by the connection with the Painlev\'e IV equation that the system (\ref{hamsys}) has the Painlev\'e property, in \cite{kecker1} an alternative proof of this statement was given. At any singularity $z_\ast$ of system (\ref{hamsys}), its solution is represented by a Laurent series, convergent in a punctured neighbourhood of $z_\ast$, of the form
\begin{equation}
\label{poleexpansion}
\begin{aligned}
q(z) =& \frac{-\rho}{z-z_\ast} + \frac{\rho z_\ast}{2} + \left( \rho \left(1+\frac{z_\ast^2}{4} \right) - \frac{\alpha}{3} + \frac{2}{3} \bar{\rho} \beta \right) (z-z_\ast) + h (z-z_\ast)^2 + \sum_{n=3}^\infty c_n (z-z_\ast)^n \\
p(z) =& \frac{\bar{\rho}}{z-z_\ast} + \frac{\bar{\rho} z_\ast}{2} + \left( \bar{\rho} \left( 1 - \frac{z_\ast^2}{4} \right) - \frac{2}{3} \rho \alpha + \frac{\beta}{3} \right) (z-z_\ast)+ k(z-z_\ast)^2 + \sum_{n=3}^\infty d_n (z-z_\ast)^n,
\end{aligned}
\end{equation}
$\rho \in \{1,\omega,\bar{\omega} \}$, having simple poles with residues given in terms of the third roots of unity. Here, $h$ and $k$ are complex parameters, coupled by the linear relation
\begin{equation*}
\rho h - k = \left( \frac{5}{4} \bar{\rho} - \frac{\alpha}{2} \rho + \frac{\beta}{2} \right) z_\ast,
\end{equation*}
so there is essentially only one free parameter. Fixing this parameter is similar to prescribing initial conditions for the system of equations, and we will see how this is done at the end of this article. 

The $3$-fold structure of the solutions of the system will be important in the following: it will allow us to construct the space of initial conditions for the system in a symmetric manner. The proof in \cite{kecker1} of the Painlev\'e property of system (\ref{hamsys}) relies on the fact that at any simple pole of the system (\ref{hamsys}) the function
\begin{equation*}
W(z) = H(z,p(z),q(z)) + \frac{p(z)^2}{q(z)}
\end{equation*}
remains finite. This in turn relies on the fact that $W$ satisfies the first-order linear differential equation
\begin{equation*}
W' + 3 \frac{p}{q^2} W = \beta \frac{p}{q} + 2 \alpha \left( \frac{p}{q} \right)^2 + 3 \left( \frac{p}{q} \right)^3,
\end{equation*}
and Lemma 6 in \cite{kecker2}, showing that the coefficients in this equation, i.e. $\frac{p}{q^2}$ and powers of $\frac{p}{q}$, remain bounded at a singularity. The auxiliary function $W$ will also play an important role below when we are showing that certain points at infinity in the space of initial conditions cannot be reached by analytic continuation of a solution. Due to the nature of the blow-up computations the expressions we are dealing with become somewhat long and we have used \textsc{Mathematica} to perform the symbolic calculations.

\section{Constructing the space of initial conditions}
At the movable singularities of a solution, the system of equations itself is well-defined and nothing can be said in general about the nature of the solution in a neighbourhood of this point just from the structure of the equation. To obtain some information on how the solution behaves in the vicinity of a movable singularity one has to include the points at infinity of the phase space of the system as the solution will be unbounded in this space. Thus the first step in constructing the space of initial conditions is to extend the system of equations in the variables $(q,p) \in \mathbb{C}^2$ to a compact space which includes the points where either $p$ or $q$ (or both) are infinite. In general any rational surface can serve as compactification but in the following we compactify the phase space of the Hamiltonian system to $\mathbb{CP}^2$. To this end we express the system of equations in the three standard coordinate charts of complex projective space, denoted by $(p,q)$, $(u_1,u_2)$ and $(v_1,v_2)$, where
\begin{equation*}
[1:q:p] = [u_1:1:u_2] = [v_1:v_2:1],
\end{equation*}
which together cover $\mathbb{CP}^2$. The sets of points $u_1 = 0$ in the variable $(u_1,u_2)$ and $v_1 = 0$ in the variables $(v_1,v_2)$ represent the line at infinity of $\mathbb{CP}^2$, denoted by $L$ in the following. In these two coordinate charts, the system of equations becomes
\begin{equation}
\label{usystem}
u_1' = -\alpha u_1^2 - z u_1 - u_2^2, \quad u_2' = -\beta u_1 - \gamma u_1 u_2 - 2 z u_2 - \frac{u_2^2+1}{u_1}
\end{equation}
and, respectively,
\begin{equation}
\label{vsystem}
v_1' = \beta v_1^2 + z v_1 + v_2^2, \quad v_2' = \alpha v_1 + \beta v_1 v_2 + 2 z v_2 + \frac{v_2^3+1}{v_1}.
\end{equation}
We see that at the points $(u_1,u_2), (v_1,v_2) \in \{(0,-1), (0,-\omega), (0,-\bar{\omega}) \}$ the right hand side of the second equation in (\ref{usystem}) or (\ref{vsystem}) becomes indeterminate, i.e.\ is of the form $\frac{0}{0}$ and nothing can be said about the behaviour of the solutions. These are the {\it base points} of the system extended on $\mathbb{CP}^2$. However, the pairs of coordinates $(u_1,u_2) = (v_1,v_2) = (0,-1)$, $(u_1,u_2) = (0,-\omega), (v_1,v_2) = (0,-\bar{\omega})$ and $(u_1,u_2) = (0,-\bar{\omega})$, $(v_1,v_2) = (0,-\omega)$ each describe the same point in $\mathbb{CP}^2$, so there are only three base points for our system. The following Lemma shows that, apart from these three base points, the line at infinity cannot be reached by analytic continuation of a solution in the complex plane.
\begin{lemma}
\label{lemL}
Let $\Gamma$ be a rectifiable path in the complex plane with endpoint $z_\ast$ such that analytic continuation of a solution of the system in the variables $(u_1,u_2)$ or $(v_1,v_2)$ along $\Gamma$ leads to a point $P \in L$, represented by the coordinates $(u_1,u_2) = (0,c)$ or $(v_1,v_2) = (0,c)$, $c \in \mathbb{C}$, respectively. Then we have $c \in \{-1,-\omega,-\bar{\omega}\}$.
\end{lemma}
\begin{proof}
Let $W_{(u)}(z)$ and $W_{(v)}(z)$ denote the functions obtained by re-writing the auxiliary function $W(z)$ in the variables $(u_1,u_2)$ and $(v_1,v_2)$, respectively. We perform the analysis for $W_{(u)}$, the case for $W_{(v)}$ being similar,
\begin{equation*}
W_{(u)}(z) = \frac{1+3 \beta  u_1^2+3 z u_1 u_2+3 \alpha  u_1^2 u_2+3 u_1^2 u_2^2+u_2^3}{3 u_1^3}.
\end{equation*}
Note that for any point $P \in L$, not one of the three base points, the values of $W_{(u)}$ and $W_{(v)}$ are infinite. At the base points themselves $W_{(u)}$ and $W_{(v)}$ are of the indeterminate form $\frac{0}{0}$. Now consider the logarithmic derivative
\begin{equation*}
\frac{d}{d z} \log(W_{(u)}(z)) = \frac{W_{(u)}'(z)}{W_{(u)}(z)} = -\frac{3 u_1 \left(u_2+2 \beta  u_1^2 u_2+3 z u_1 u_2^2+\alpha  u_1^2 u_2^2+u_2^4\right)}{1+3 \beta  u_1^2+3 z u_1 u_2+3 \alpha u_1^2 u_2+3 u_1^2 u_2^2+u_2^3}.
\end{equation*}
Suppose that $P$ has coordinates $(0,c)$ where $c \notin \{-1,-\omega,-\bar{\omega}\}$. In some neighbourhood $U$ of $P$ the logarithmic derivative is bounded, say by some number $M$. Let $z_0 \in \Gamma$ denote a point on the curve such that $W_{(u)}(z_0) \neq 0$ and $(u_1(z),u_2(z)) \in U$ for all $z \in \Gamma_{z_0}$, the part of $\Gamma$ following $z_0$. Integrating along the path $\Gamma_{z_0}$ shows that 
\begin{equation*}
|\log(W_{(u)}(z_\ast))| = \Big| \log(W_{(u)}(z_0)) + \int_{z_0}^{z_\ast} \frac{d}{d z} \log(W_{(u)}(z)) dz \Big| \leq |\log(W_{(u)}(z_0))| + \int_{z_0}^{z_\ast} M |dz| < \infty.
\end{equation*}
On the other hand, since $(u_1(z_\ast),u_2(z_\ast)) = P$, we would have $W_{(u)}(z_\ast)$ infinite, contradicting the above. Hence we must have $P=(0,-1)$, $P=(0,-\omega)$ or $P=(0,-\bar{\omega})$. 
\end{proof}

We will now describe the procedure of blowing up the surface at the points $(u_1,u_2) = (0,-\rho)$, where $\rho \in \{1,\omega,\bar{\omega}\}$. The analysis for all three points is similar, differing only in various factors of $\omega$ and $\bar{\omega}$. The blow-up at a point $P=(p_1,p_2) \in \mathbb{C}^2$ in the coordinates $(u_1,u_2)$ is defined by the following construction,
\begin{equation*}
\text{Bl}_P = \{ ((u_1,u_2),[\zeta_1,\zeta_2]) \in \mathbb{C}^2 \times \mathbb{CP}^1 : (u_1-p_1) \zeta_2 = (u_2-p_2) \zeta_1 \},
\end{equation*}
where $[\zeta_1:\zeta_2]$ are homogeneous coordinates on $\mathbb{CP}^1$. Note that a blow-up is a local operation on the space when seen as follows. We define the projection $\pi: \text{Bl}_P \to \mathbb{C}^2$ by $((u_1,u_2),[\zeta_1,\zeta_2]) \mapsto (u_1,u_2)$. For any point $Q \neq P$, the pre-image $\pi^{-1}(Q)$ consists of a single point whereas the pre-image of $P$ itself is $\pi^{-1}(P) = P \times \mathbb{CP}^1$. This is called the exceptional curve in $\text{Bl}_P$. So we see that by a blow-up at $P$ the point itself becomes inflated to a sphere whereas away from $P$ the local geometry of the space does not change. In the coordinates $(u_1,u_2)$ the blow-up is performed by introducing two new coordinate charts, denoted by $(u_{1,1},u_{1,2})$ and $(u_{2,1},u_{2,2})$, respectively. The first new coordinate chart is given by
\begin{equation*}
u_{1,1} = \frac{\zeta_1}{\zeta_2} = \frac{u_1 - p_1}{u_2 - p_2}, \quad u_{1,2} = u_2 - p_2,
\end{equation*}
and covers the part of $\text{Bl}_P$ where $\zeta_2 \neq 0$, whereas the second chart is
\begin{equation*}
u_{2,1} = u_1 - p_1, \quad u_{2,2} = \frac{\zeta_2}{\zeta_1} = \frac{u_2 - p_2}{u_1 - p_1},
\end{equation*}
covering the part of $\text{Bl}_P$ where $\zeta_1 \neq 0$. The exceptional curves introduced by blowing up each base point will be denoted $L_1^{(\rho)}$, $\rho \in \{1,\omega,\bar{\omega}\}$, and are parametrised by $(u_{1,1},u_{1,2}) = (c,0)$ in the first chart and $(u_{2,1},u_{2,2}) = (0,c)$ in the second chart, $c \in \mathbb{C}$. When looking for new base points after performing the blow-up, we only need to look on the exceptional curve. Re-written in the variables after the blow-up, the system of equations in the first chart becomes
\begin{equation*}
\begin{aligned}
u_{1,1}' &= \frac{2\bar{\rho} - 2 \rho z u_{1,1}}{u_{1,2}} + (\beta - \rho \alpha)u_{1,1}^2 + z u_{1,1} - \rho, \\
u_{1,2}' &= (\rho \alpha-\beta) u_{1,1} u_{1,2} - \alpha u_{1,1} u_{1,2}^2 + 2z(\rho - u_{1,2}) - \frac{u_{1,2}^2 - 3\rho u_{1,2} + 3\bar{\rho}}{u_{1,1}},
\end{aligned}
\end{equation*}
where it is indeterminate at the point $(u_{1,1},u_{1,2}) = \left( \frac{\rho}{z},0 \right)$. In the second chart,
\begin{equation*}
\begin{aligned}
u_{2,1}' &= -\bar{\rho} - z u_{2,1} - \alpha u_{2,1}^2 + 2\rho u_{2,1} u_{2,2} - u_{2,1}^2 u_{2,2}^2 \\
u_{2,2}' &= \rho \alpha - \beta - zu_{2,2} + \rho u_{2,2}^2 + \frac{2\rho z - 2\bar{\rho} u_{2,2}}{u_{2,1}},
\end{aligned}
\end{equation*}
with indeterminacy at $(u_{2,1},u_{2,2}) = ( 0, \bar{\rho}z )$. Since $u_{1,1} = u_{2,2}^{-1}$ we see that the base points of these systems in fact correspond to the same point on the exceptional curve, so there is only one new base point. Also note that the location of the base point becomes $z$ dependent. We will now show that a solution cannot pass through the exceptional curve except at the base points.
\begin{lemma}
\label{lemL1}
Let $\Gamma$ be a rectifiable path in the complex plane with end point $z_\ast$ such that analytic continuation of a solution along $\Gamma$ leads to a point $P$ on one of the exceptional curves $L_1^{(\rho)}$, $\rho \in \{1, \omega, \bar{\omega} \}$. Let $P$ have coordinates $(u_{1,1},u_{1,2}) = (c^{-1},0)$ in the first chart and $(u_{2,1},u_{2,2}) = (0,c)$, $c \in \mathbb{C}$, in the second chart, respectively. Then we must have $c = \bar{\rho} z_\ast$.
\end{lemma}
\begin{proof}
The proof runs along the same lines as Lemma \ref{lemL}, by considering the auxiliary function $W$ re-written in the variables $(u_{1,1},u_{1,2})$, denoted by $W_1$, and in the variables $(u_{2,1},u_{2,2})$, denoted $W_2$. Again, we will only perform the analysis for $W_1$, the case for $W_2$ being similar. We have
\begin{equation*}
W_1(z) = u_{1,1}^{-3} u_{1,2}^{-2} \cdot P_1(z,u_{1,1},u_{1,2}),
\end{equation*}
where
\begin{equation*}
P_1(z,u_{1,1},u_{1,2}) = \bar{\rho} - z \rho  u_{1,1} - \rho  u_{1,2} + z u_{1,1} u_{1,2}+ \bar{\rho} \left(1 - \bar{\rho} \alpha  +  \rho \beta \right) u_{1,1}^2 u_{1,2} + \frac{1}{3} u_{1,2}^2+ (\alpha - 2 \rho) u_{1,1}^2 u_{1,2}^2 + u_{1,1}^2 u_{1,2}^3.
\end{equation*}
On the exceptional curve $(u_{1,1},u_{1,2}) = (c^{-1},0)$, $c \in \mathbb{C}$, $W_1$ is infinite apart from the point with $c = \bar{\rho} z$, where it is of the indeterminate form $\frac{0}{0}$. On the other hand, the logarithmic derivative of $W_1(z)$ is
\begin{equation*}
\begin{aligned}
\frac{W_1'}{W_1} = & P_1(z,u_{1,1},u_{1,2})^{-1} \cdot 3 u_{1,1} u_{1,2} \left(3 - 3 z \bar{\rho} u_{1,1} - 6 \bar{\rho} u_{1,2} + 6 z \rho  u_{1,1} u_{1,2} + (2 \rho \beta -\bar{\rho} \alpha)  u_{1,1}^2 u_{1,2} + 4 \rho  u_{1,2}^2 - 3 z u_{1,1} u_{1,2}^2 \right. \\ & \left. + 2 (\rho \alpha - \beta) u_{1,1}^2 u_{1,2}^2 - u_{1,2}^3 - \alpha  u_{1,1}^2 u_{1,2}^3 \right),
\end{aligned}
\end{equation*}
which is bounded in a neighbourhood of any point $(u_{1,1},u_{1,2}) = (c^{-1},0)$, $c \neq \bar{\rho} z$. By a similar integral estimate as in Lemma \ref{lemL}, we obtain a contradiction that $|W_1(z_\ast)| < \infty$. Hence we must have $c = \bar{\rho} z_\ast$. 
\end{proof}

We will now perform the second blow-up, with the computations carried out for the base point $(0,\bar{\rho} z)$ in the variables $(u_{2,1},u_{2,2})$. We introduce two new coordinate charts,
\begin{equation*}
\hat{u}_{1,1} = \frac{u_{2,1}}{\hat{u}_{2,2} - \bar{\rho}z}, \quad \hat{u}_{1,2} = u_{2,2} - \bar{\rho} z,
\end{equation*}
and
\begin{equation*}
\hat{u}_{2,1} = u_{2,1}, \quad \hat{u}_{2,2} = \frac{u_{2,2} - \bar{\rho} z}{u_{2,1}}.
\end{equation*}
In these coordinates the system of equations takes the following form,
\begin{equation*}
\begin{aligned}
\hat{u}_{1,1}' = & \frac{\bar{\rho} + (\bar{\rho} + \beta - \rho \alpha) \hat{u}_{1,1}}{\hat{u}_{1,2}} + \rho \hat{u}_{1,1} \hat{u}_{1,2} -(\rho z^2 + \alpha) \hat{u}_{1,1}^2 \hat{u}_{1,2} - 2 \bar{\rho} z \hat{u}_{1,1}^2 \hat{u}_{1,2}^2  - \hat{u}_{1,1}^2 \hat{u}_{1,2}^3 \\
\hat{u}_{1,2}' = & \rho \alpha - \beta - \bar{\rho} + z \hat{u}_{1,2} + \rho \hat{u}_{1,2}^2 - \frac{2\bar{\rho}}{\hat{u}_{1,1}} \\
\hat{u}_{2,1}' = & -\rho^2 + z \hat{u}_{2,1} - (\rho z^2 + \alpha) \hat{u}_{2,1}^2 + 2\rho \hat{u}_{2,1}^2 \hat{u}_{2,2} - 2z\bar{\rho} \hat{u}_{2,1}^3 \hat{u}_{2,2} - \hat{u}_{2,1}^4 \hat{u}_{2,2}^2 \\
\hat{u}_{2,2}' = & \frac{\rho \alpha - \beta - \bar{\rho} - \bar{\rho} \hat{u}_{2,2}}{\hat{u}_{2,1}} + (\rho z^2 + \alpha) \hat{u}_{2,1} \hat{u}_{2,2} - \rho \hat{u}_{2,1} \hat{u}_{2,2}^2 + 2 z \bar{\rho} \hat{u}_{2,1}^2 \hat{u}_{2,2}^2  + \hat{u}_{2,1}^3 \hat{u}_{2,2}^3.
\end{aligned}
\end{equation*}
Still, after the second blow-up the indeterminacy in the system of equations prevails, namely in the first chart at the coordinates $(\hat{u}_{1,1},\hat{u}_{1,2}) = \left((\bar{\rho} \alpha - \rho \beta - 1)^{-1},0\right)$ and at $(\hat{u}_{2,1},\hat{u}_{2,2}) = (0,\bar{\rho} \alpha -\rho \beta - 1)$ in the second coordinate chart, these representing the same point. The exceptional curves introduced by the second blow-ups will be denoted by $L_2^{(1)}$, $L_2^{(\omega)}$, $L_2^{(\bar{\omega})}$. Similar to Lemmas \ref{lemL} and \ref{lemL1}, the next Lemma shows that the exceptional curve cannot be reached by analytic continuation of a solution except at the base points.
\begin{lemma}
\label{lemL2}
Let $\Gamma$ be a rectifiable path in the complex plane with endpoint $z_\ast$ such that analytic continuation of a solution along $\Gamma$ leads to a point $P$ on one of the exceptional curves $L_2^{(\rho)}$, $\rho \in \{1,\omega,\bar{\omega}\}$. Let $P$ have coordinates $(\hat{u}_{1,1},\hat{u}_{1,2}) = (c^{-1},0)$ and $(\hat{u}_{2,1},\hat{u}_{2,2}) = (0,c)$. Then we must have $c= \bar{\rho} \alpha - \rho \beta -1$.
\end{lemma}
\begin{proof}
Again we consider the auxiliary function $W$, re-written in the variables $(\hat{u}_{1,1},\hat{u}_{1,2})$ and $(\hat{u}_{2,1},\hat{u}_{2,2})$, denoted $\hat{W}_1$ and $\hat{W}_2$, respectively. Again we only consider the case $\hat{W}_1$,
\begin{equation*}
\hat{W}_1 = \hat{u}_{1,1}^{-2} \hat{u}_{1,2}^{-1}  \cdot \hat{P}_1(z,\hat{u}_{1,1},\hat{u}_{1,2}),
\end{equation*}
where
\begin{equation*}
\begin{aligned}
\hat{P}_1 = & \Big( \bar{\rho}+ (\beta - \rho \alpha + \bar{\rho}) \hat{u}_{1,1} - z \hat{u}_{1,1} \hat{u}_{1,2} + \left( (\bar{\rho} \alpha - 2) z + z^3/3 \right) \hat{u}_{1,1}^2 \hat{u}_{1,2} - \rho \hat{u}_{1,1} \hat{u}_{1,2}^2+ ( \alpha - 2\rho + \rho z^2)   \hat{u}_{1,1}^2 \hat{u}_{1,2}^2  \\ & + z^2 \rho  \hat{u}_{1,1}^3 \hat{u}_{1,2}^2+ z \bar{\rho} \hat{u}_{1,1}^2 \hat{u}_{1,2}^3+2 z \bar{\rho} \hat{u}_{1,1}^3 \hat{u}_{1,2}^3+ \frac{1}{3} \hat{u}_{1,1}^2 \hat{u}_{1,2}^4+ \hat{u}_{1,1}^3 \hat{u}_{1,2}^4 \Big).
\end{aligned}
\end{equation*}
We note that on the exceptional curve $(\hat{u}_{1,1},\hat{u}_{1,2}) = (c^{-1},0)$, $c \in \mathbb{C}$, $\hat{W}_1$ is infinite, apart from at the base point $c = \bar{\rho} \alpha - \rho \beta -1$, where it is of the indeterminate form $\frac{0}{0}$. The logarithmic derivative of $\hat{W}_1$ is given by
\begin{equation*}
\begin{aligned}
\frac{\hat{W}_1'}{\hat{W}_1} = & \hat{P}_1(z,\hat{u}_{1,1},\hat{u}_{1,2})^{-1} \cdot \hat{u}_{1,1} \hat{u}_{1,2} \left(3 + (2 \rho \beta -  \bar{\rho} \alpha) \hat{u}_{1,1} - 6 z \rho  \hat{u}_{1,1} \hat{u}_{1,2} - ( 2z(\bar{\rho} \beta - \alpha) - \rho z^3) \hat{u}_{1,1}^2 \hat{u}_{1,2} - 6 \bar{\rho} \hat{u}_{1,1} \hat{u}_{1,2}^2 \right. \\ & \left. + (2 \rho \alpha - 2 \beta + 6 \bar{\rho} z^2) \hat{u}_{1,1}^2 \hat{u}_{1,2}^2 - ( \rho \alpha z^2 + \bar{\rho} z^4) \hat{u}_{1,1}^3 \hat{u}_{1,2}^2 + 9 z \hat{u}_{1,1}^2 \hat{u}_{1,2}^3 - (2 \bar{\rho} \alpha z + 4z^3) \hat{u}_{1,1}^3 \hat{u}_{1,2}^3  + 4 \rho  \hat{u}_{1,1}^2 \hat{u}_{1,2}^4 \right. \\ & \left. - (\alpha + 6z^2)  \hat{u}_{1,1}^3 \hat{u}_{1,2}^4 - 4 z \bar{\rho} \hat{u}_{1,1}^3 \hat{u}_{1,2}^5 - \hat{u}_{1,1}^3 \hat{u}_{1,2}^6\right).
\end{aligned}
\end{equation*}
Again, in a neighbourhood of any point on the exceptional curve other than the base point, the logartihmic derivative of $\hat{W}_1$ is bounded. By an integral estimate similar to the one in Lemmas \ref{lemL} and \ref{lemL1} it follows, by analytic continuation along $\Gamma$, that $|\hat{W}_1(z_\ast)| < \infty$, in contradiction to the fact that $\hat{W}_1$ is infinite there. Hence the solution must run into the base point $(u_{1,1},u_{1,2}) = (( \bar{\rho} \alpha - \rho \beta -1)^{-1},0)$.
\end{proof}
We will now show that one further blow-up for each base point will resolve the indeterminacy in the system of equations so that one obtains a regular initial value problem. We will perform the blow-up in the variables $(\hat{u}_{2,1},\hat{u}_{2,2})$. For this we again introduce two new coordinate charts,
\begin{equation*}
\tilde{u}_{1,1} = \frac{\hat{u}_{2,1}}{\hat{u}_{2,2} + 1 - \bar{\rho} \alpha + \rho \beta}, \quad \tilde{u}_{1,2} = \hat{u}_{2,2} + 1 - \bar{\rho} \alpha + \rho \beta,
\end{equation*}
and
\begin{equation*}
\tilde{u}_{2,1} = \hat{u}_{2,1}, \quad \tilde{u}_{2,2} = \frac{\hat{u}_{2,2} + 1 - \bar{\rho} \alpha + \rho \beta}{\hat{u}_{2,1}}.
\end{equation*}
The equations after the third blow-up read, in the variables $(\tilde{u}_{1,1},\tilde{u}_{1,2})$,
\begin{equation*}
\begin{aligned}
\tilde{u}_{1,1}' = & z \tilde{u}_{1,1}+ \rho  \left(1 + z^2 + \beta \rho \right) \left(1 - \bar{\rho} \alpha + \rho \beta \right) \tilde{u}_{1,1}^2 + 2\left( \alpha - 2 \rho - z^2 \rho - 2 \beta  \bar{\rho}\right) \tilde{u}_{1,1}^2 \tilde{u}_{1,2}+3 \rho \tilde{u}_{1,1}^2 \tilde{u}_{1,2}^2 \\ & -2z \bar{\rho} (1- \bar{\rho} \alpha + \rho \beta)^2 \tilde{u}_{1,1}^3 \tilde{u}_{1,2}+ 6z \bar{\rho} \left(1 - \bar{\rho} \alpha + \rho \beta \right) \tilde{u}_{1,1}^3 \tilde{u}_{1,2}^2-4 z \bar{\rho} \tilde{u}_{1,1}^3 \tilde{u}_{1,2}^3 +(1-\bar{\rho} \alpha + \rho \beta)^3 \tilde{u}_{1,1}^4 \tilde{u}_{1,2}^2 \\ & -4 (1- \bar{\rho} \alpha + \rho \beta)^2 \tilde{u}_{1,1}^4 \tilde{u}_{1,2}^3+ 5\left(1 - \bar{\rho} \alpha + \rho \beta \right) \tilde{u}_{1,1}^4 \tilde{u}_{1,2}^4-2 \tilde{u}_{1,1}^4 \tilde{u}_{1,2}^5 \\
\tilde{u}_{1,2}' = & -\frac{\bar{\rho}}{\tilde{u}_{1,1}} -\rho \left(1 + z^2+ \rho \beta \right) \left( 1 - \bar{\rho} \alpha + \rho \beta \right) \tilde{u}_{1,1} \tilde{u}_{1,2}+\left(-\alpha +2 \rho +z^2 \rho +2 \beta  \bar{\rho}\right) \tilde{u}_{1,1} \tilde{u}_{1,2}^2-\rho  \tilde{u}_{1,1} \tilde{u}_{1,2}^3 \\ & + 2z \bar{\rho} (1 - \bar{\rho} \alpha + \rho \beta)^2 \tilde{u}_{1,1}^2 \tilde{u}_{1,2}^2 -4z\bar{\rho} \left(1 - \bar{\rho} \alpha + \rho \beta \right) \tilde{u}_{1,1}^2 \tilde{u}_{1,2}^3+2 z \bar{\rho} \tilde{u}_{1,1}^2 \tilde{u}_{1,2}^4 -(1-\bar{\rho} \alpha + \rho \beta)^3 \tilde{u}_{1,1}^3 \tilde{u}_{1,2}^3 \\ & + 3(1- \bar{\rho} \alpha + \rho \beta)^2 \tilde{u}_{1,1}^3 \tilde{u}_{1,2}^4 -3 \left(1 -\bar{\rho} \alpha + \rho \beta \right) \tilde{u}_{1,1}^3 \tilde{u}_{1,2}^5+ \tilde{u}_{1,1}^3 \tilde{u}_{1,2}^6,
\end{aligned}
\end{equation*}
and in the variables $(\tilde{u}_{2,1},\tilde{u}_{2,2})$,
\begin{equation}
\label{finalsys}
\begin{aligned}
\tilde{u}_{2,1}' = & -\bar{\rho}+z \tilde{u}_{2,1}+\left(\alpha -2 \rho -z^2 \rho -2 \beta  \bar{\rho}\right) \tilde{u}_{2,1}^2+\left(1 - \bar{\rho} \alpha + \rho \beta \right) \tilde{u}_{2,1}^5 \tilde{u}_{2,2}-\tilde{u}_{2,1}^6 \tilde{u}_{2,2}^2+ 2z \bar{\rho} \left(1- \bar{\rho} \alpha + \rho \beta \right) \tilde{u}_{2,1}^3 \\ & + 2 \rho \tilde{u}_{2,1}^3 \tilde{u}_{2,2} -(1-\bar{\rho} \alpha + \rho \beta)^2 \tilde{u}_{2,1}^4 - 2 z \bar{\rho} \tilde{u}_{2,1}^4 \tilde{u}_{2,2} \\
\tilde{u}_{2,2}' = & -\rho (1 + z^2 + \rho \beta)(1-\bar{\rho}\alpha+\rho \beta) -z \tilde{u}_{2,2} -5 \left(1 - \bar{\rho} \alpha + \rho \beta \right) \tilde{u}_{2,1}^4 \tilde{u}_{2,2}^2+2 \tilde{u}_{2,1}^5 \tilde{u}_{2,2}^3 + 2z \bar{\rho} (1 - \bar{\rho} \alpha + \rho \beta)^2 \tilde{u}_{2,1} \\ & + \left(-2 \alpha +4 \rho +2 z^2 \rho +4 \beta  \rho ^2\right) \tilde{u}_{2,1} \tilde{u}_{2,2} - (1-\bar{\rho} \alpha + \rho \beta)^3 \tilde{u}_{2,1}^2 - 6z \bar{\rho} (1 -\bar{\rho} \alpha + \rho \beta) \tilde{u}_{2,1}^2 \tilde{u}_{2,2} -3 \rho \tilde{u}_{2,1}^2 \tilde{u}_{2,2}^2 \\ & +4(1-\bar{\rho} \alpha + \rho \beta)^2 \tilde{u}_{2,1}^3 \tilde{u}_{2,2}+4z \bar{\rho} \tilde{u}_{2,1}^3 \tilde{u}_{2,2}^2.
\end{aligned}
\end{equation}
From these equation we see that on the exceptional curves $L_3^{(1)}$, $L_3^{(\omega)}$ and $L_3^{(\bar{\omega})}$, introduced by the third blow-ups of each base point, the system of equations becomes a regular initial value problem in the variables $(\tilde{u}_{2,1},\tilde{u}_{2,2})$. If we denote by $\mathcal{S}$ the compact space obtained by the three blow-ups at each base point, covered by all the coordinate systems introduced in the process, the system describes a regular intial value problem on the space
\begin{equation*}
\mathcal{I} = \mathcal{S} \setminus \left( L \cup L_1^{(1)} \cup L_1^{(\omega)} \cup L_1^{(\bar{\omega})} \cup L_2^{(1)} \cup L_2^{(\omega)} \cup L_2^{(\bar{\omega})} \right),
\end{equation*} 
this is the space of initial conditions. The changes of variables introduced by the three blow-ups amount to the following relationship to the original variables $(p,q)$,
\begin{equation*}
\tilde{u}_{1,1}(z) = \frac{1}{q(z) r(z)}, \quad \tilde{u}_{1,2}(z) = r(z), \quad \tilde{u}_{2,1}(z) = \frac{1}{q(z)}, \quad \tilde{u}_{2,2}(z) = q(z) r(z),
\end{equation*}
where $r(z) = 1 - \bar{\rho} \alpha + \rho \beta - \bar{\rho} z q(z) + \rho q(z)^2 + q(z) p(z)$. These bi-rational relations can be inverted easily to yield
\begin{equation*}
\begin{aligned}
q(z) &= \frac{1}{\tilde{u}_{2,1}(z)} \\
p(z) &= -\frac{\rho}{\tilde{u}_{2,1}(z)} + \bar{\rho} z - (1-\bar{\rho} \alpha + \rho \beta ) \tilde{u}_{2,1} + \tilde{u}_{2,1}^2 \tilde{u}_{2,2}.
\end{aligned}
\end{equation*}
We can seek local analytic solutions of the final system (\ref{finalsys}), with initial conditions $(\tilde{u}_{2,1}(z_\ast),\tilde{u}_{2,2}(z_\ast)) = (0,c)$ on the exceptional curve, in the form of power series
\begin{equation*}
\tilde{u}_{2,1}(z) = \sum_{n=1}^\infty a_n(z-z_\ast)^n, \quad \tilde{u}_{2,2}(z) = c + \sum_{n=1}^\infty b_n(z-z_\ast)^n,
\end{equation*}
where all coefficients $a_n$, $b_n$, $n=1,2,3,\dots$, can be obtained recursively. One finds
\begin{equation*}
a_1 = -\bar{\rho}, \quad a_2 = -\frac{z_\ast \bar{\rho}}{2}, \quad a_3 = \frac{\rho \alpha - 2\beta}{3} - \bar{\rho} \left( 1 + \frac{z_\ast^2}{2} \right), \quad a_4 = -\frac{c \rho}{2} + \left(\frac{5\alpha \rho}{6} - \frac{7\beta}{6} - \frac{15\bar{\rho}}{8} \right) z_*-\frac{3}{8} \bar{\rho } z_*^3, \quad \cdots 
\end{equation*}
\begin{equation*}
\begin{aligned}
b_1 &= \alpha - \beta^2 - \rho + \alpha \beta \rho - 2 \beta \bar{\rho} - c z_\ast + (\alpha - \bar{\rho} \beta - \rho) z_\ast^2, \\ b_2 &= c \left(-\frac{5}{2} -2 \beta \rho + \alpha  \bar{\rho} \right) +\frac{1}{2} \left(5 \alpha -\beta ^2-3 \rho +3 \alpha  \beta  \rho -2 \alpha ^2 \bar{\rho }-4 \beta  \bar{\rho }\right) z_\ast-\frac{c z_\ast^2}{2}, \quad \cdots
\end{aligned}
\end{equation*}
Thus a solution in the variables $(\tilde{u}_{2,1},\tilde{u}_{2,2})$, locally analytic in the neighbourhood of a point on one of the exceptional curves $L_3^{(\rho)}$, $\rho \in\{1,\omega,\bar{\omega}\}$, where $\tilde{u}_{2,1}=0$, becomes a simple pole in the original variables $(p,q)$, corresponding to the expansions (\ref{poleexpansion}). The parameters $h$ and $k$ in the expansions (\ref{poleexpansion}) are determined in terms of $\tilde{u}_{2,2}(z_\ast)=c$, the position of the initial point on the exceptional curve $L_3^{(\rho)}$, by the expressions
\begin{equation}
\label{hkrelation}
h = \frac{c}{2} + \left(-\frac{\alpha}{2} + \frac{7 \rho}{8} + \frac{\beta \bar{\rho}}{2} \right) z_\ast, \quad k = \frac{c \rho}{2} - \frac{3 \bar{\rho} z_\ast}{8}.
\end{equation}

Although the space $\mathcal{I}$ itself is not compact, Lemmas \ref{lemL}, \ref{lemL1} and \ref{lemL2} show that a solution, when analytically continued along some path in the complex plane cannot pass through the line at infinity $L$ or any of the exceptional curves $L_1^{(\rho)}$, $L_2^{(\rho)}$, $\rho \in \{1,\omega,\bar{\omega}\}$. We have thus shown the following theorem by which we conclude this article.
\begin{thm}
Let $(p(z),q(z))$ be a local analytic solution of the system (\ref{hamsys}) in a neighbourhood of a point $z_0 \in \mathbb{C}$. Let $\Gamma$ be a rectifiable path from $z_0$ to some point $z_\ast$ such that $(p(z),q(z))$ can be analytically continued along $\Gamma$ up to, but not including the point $z_\ast$. Then, in some coordinate chart of $\mathcal{I}$, the solution, re-written in these coordinates, can be analytically continued to $z_\ast$ leading to a point $P \in L_3^{(\rho)}$, $\rho \in \{1,\omega,\bar{\omega} \}$, not covered by the original coordinate chart $(p,q)$. The local analytic solution about $P$ corresponds to a simple pole of the form (\ref{poleexpansion}) in the variables $(p,q)$ with the parameters $h$ and $k$ fixed by the location of $P$ via the expressions (\ref{hkrelation}).
\end{thm}

\bibliographystyle{plain}

\noindent
Thomas Kecker \\
Department of Mathematics \\
University of Portsmouth \\
Lion Gate Building, Lion Terrace \\
Portsmouth, PO1 3HF \\
United Kingdom \\
thomas.kecker@port.ac.uk

\end{document}